\documentclass[12pt]{amsart}
\usepackage{fullpage}
\usepackage{fourier}
\usepackage{amssymb}
\usepackage[makeroom]{cancel}
\usepackage{amsthm,amsfonts,amsmath,mathrsfs,amssymb}
\usepackage{tikz}                  
\usepackage{enumerate}
\usepackage{enumitem}
\usepackage[T1]{fontenc}
\usepackage{dsfont}
\usepackage[utf8]{inputenc}
\usepackage[english]{babel}
\usetikzlibrary{matrix}             

\newcommand{\T}{\mathbb{T}}

\newcommand{\Ht}{\mathscr{H}^2}

\newtheorem{thm}{Theorem}
\newtheorem*{thm*}{Theorem}

\newtheorem{lem}{Lemma}
\newtheorem*{lem*}{Lemma}

\newtheorem{prop}{Proposition}
\newcommand{\Da}{{\mathscr{D}}_\alpha}

\newcommand{\Hw}{{\mathscr{H}}_w}
\newcommand{\Cf}{{\mathscr{C}}_\Phi}
\newcommand{\D}{\mathbb D}
\newcommand{\C}{\mathbb C}
\newcommand{\N}{\mathbb N}
\newcommand{\R}{\mathbb R}
\newcommand{\Z}{\mathbb Z}

\begin{document}
\author{Jing Zhao} \address{Jing Zhao\\Department of Mathematical Sciences \\
Norwegian University of Science and Technology \\
NO-7491 Trondheim \\
Norway} \email{jing.zhao@ntnu.no; jingzh95@gmail.com}
\title{Iteration of composition operators  \\  on small Bergman spaces of Dirichlet series}
  \thanks{The author is supported by the Research Council of Norway grant 227768.}
\maketitle
\begin{abstract}
The Hilbert spaces $\mathscr{H}_{w}$ consisiting of Dirichlet series  $F(s)=\sum_{ n = 1}^\infty a_n n^{ -s }$ that satisfty $\sum_{ n=1 }^\infty  | a_n |^2/ w_n < \infty$, with $\{w_n\}_n$ of average order $\log_j n$ (the $j$-fold logarithm of $n$), can be embedded into certain small Bergman spaces. Using this embedding, we study the Gordon--Hedenmalm theorem on such $\Hw$ from an iterative point of view. By that theorem, the composition operators are generated by functions of the form $\Phi(s) = c_0s + \phi(s)$, where $c_0$ is a nonnegative integer and $\phi$ is a Dirichlet series with certain convergence and mapping properties. The iterative phenomenon takes place when $c_0=0$. It is verified for every integer $j\geqslant 1$, real $\alpha>0$ and $\{w_n\}_{n}$ having average order $(\log_j^+ n)^\alpha$ , that the composition operators map $\Hw$ into a scale of $\mathscr{H}_{w'}$ with $w_n'$ having average order $( \log_{j+1}^+n)^\alpha$. The case $j=1$ can be deduced from the proof of the main theorem of a recent paper of Bailleul and Brevig, and we adopt the same method to study the general iterative step.
\end{abstract}

\section{introduction}

Let $\Ht$ be the Hilbert space of Dirichlet series with square summable coefficients. A theorem of Gordon and Hedenmalm \cite{gh} classifies the set of analytic functions $\Phi: \C_{1/2}\to\C_{1/2}$ which generate composition operators that map $\Ht$ into itself. Let $\C_\theta$ denote the half-plane $\C_\theta:= \left\{ s=\sigma+it: \sigma>\theta \right\}$. The Gordon--Hedenmalm theorem reads as follows, in a slightly strengthened version found in \cite{strco}.
\begin{thm}[Gordon--Hedenmalm Theorem]
A function $\Phi: \C_{1/2}\to\C_{1/2}$ generates a bounded composition operator $\mathscr{C}_{\Phi}:\Ht\to\Ht$ if and only if $\Phi$ is of the form
\begin{equation}\label{phi}
\Phi(s) = c_0s + \sum_{n=1}^\infty c_n n^{-s} = c_0s + \phi(s),
\end{equation}
where $c_0$ is a nonnegative integer, and $\phi$ is a Dirichlet series that converges uniformly in $\C_\epsilon$ ($\epsilon>0$) with the following mapping properties:
\begin{enumerate}[label=(\roman*)]
\item If $c_0=0$, then $\phi(\C_0) \subset \C_{1/2}$;
\item If $c_0\geqslant 1$, then either $\phi\equiv0$ or $\phi(\C_0) \subset \C_0$.
\end{enumerate}
\end{thm}
The set of such $\Phi$ is called the Gordon--Hedenmalm class and denoted by $\mathscr G$. With the same convergence and mapping properties, the Gordon--Hedenmalm theorem was extended to the weighted Hilbert spaces $\Da$ which consists of Dirichlet series $F(s)=\sum_{ n = 1}^\infty a_n n^{ -s }$ that satisfy  $\sum_{ n=1 }^\infty  | a_n |^2 / d(n)^\alpha < \infty$  in \cite{bb}. Here $d(n)$ is the divisor function which counts the number of divisors of $n$ and $\alpha>0$. In particular, for $c_0=0$, the composition operators map $\Da$ into $\mathscr D_\beta$  with $\beta=2^\alpha-1$.
It should be noticed that $\mathscr D_\beta$ are spaces that are smaller than $\Da$ when $0<\alpha<1$ and bigger when $\alpha>1$. 

We observe from the proof of this extension (see \cite[Theorem 1]{bb}) that $\Da$ are actually mapped into weighted Hilbert spaces that consist of Dirichlet series $F(s)=\sum_{ n = 1}^\infty a_n n^{ -s }$ satisfying  \[\sum_{ n=1 }^\infty  | a_n |^2 / (1+\Omega(n))^\beta < \infty,\] where $\Omega(n)$ is the number of prime factors of $n$ (counting multiplicities). We say that an arithmetic function $f$ has average order $g$ if $\frac{ 1 }{ X } \sum_{n\leqslant X} f(n) \asymp g(X) $. Since $d(n)^\alpha$ has average order $(\log n)^{\beta}$ and $\Omega(n)^\beta$ has average order $(\log\log n)^\beta$ (see Proposition \ref{higherorder}), $\Da$ are in fact mapped into smaller weighted Hilbert spaces.  In this paper, we show that the phenomenon of gaining one more fold of the logarithm persists for more general weights that have average order $(\log_j n)^\alpha$ with $j\in\N$ and real $\alpha>0$. 

Let $\log_j x$ denote the $j$-fold logarithm of $x$ so that $\log_1 x=\log x$ and $\log_j x=\log_{j-1}(\log x)$. For convenience, we define $\log_0 x :=x$ and
\[\log^+ |x| := \max \big\{ |x|,0 \big\};\qquad \log_j^+|x| := \log^+\left(\log_{j-1}\right)^+|x|,\quad j\geqslant 2. \] 
Define
\begin{equation}
\label{hw}
 \Hw := \left\{ F(s)=\sum_{ n=1 }^\infty a_n n^{ -s } :\quad \| F  \|_{ \Hw }^2 := \sum_{ n=1 }^\infty \frac{ | a_n |^2 }{ w_n } < \infty \right\}.
\end{equation}
For every real number $\alpha>0$ and integer $j\geqslant 0$, let
\[\mathscr{H}_{\log,0} := \left\{ F(s)=\sum_{ n=1 }^\infty a_n n^{ -s } :\quad  \| F  \|_{  \mathscr H_{\log, j} }^2 = \sum_{ n=1 }^\infty \frac{ | a_n |^2 }{ \left( 1 + \Omega(n)  \right)^\alpha } < \infty \quad  \right\} =:\mathscr H_\Omega\]
\[  \mathscr H_{\log, j} := \left\{ F(s)=\sum_{ n=1 }^\infty a_n n^{ -s } :\quad  \| F  \|_{  \mathscr H_{\log, j} }^2 = \sum_{ n=1 }^\infty \frac{ | a_n |^2 }{ \left( 1 + \log_{j}^+ \Omega(n)  \right)^\alpha } < \infty \quad  \right\},\quad j\geqslant 1.\]
Our main result reads as follows.
\begin{thm}\label{main}
Let $\alpha>0$ be a real number and $j\geqslant 1$ be an integer.  When the weight $\{w_n\}_n$ of $\Hw$ has average order $(\log_j^+ n)^\alpha$, a function $\Phi: \C_{1/2} \to \C_{1/2}$ of the form defined in \eqref{phi} with $c_0=0$ generates a composition operator $\Cf: \Hw \to \mathscr H_{\log, j-1}$ if and only if $\Phi\in\mathscr G$.
\end{thm} 
There are a few things we should make clear. First, it is proved in Section \ref{averageorder} that the average order of $(\log_j^+\Omega(n))^\alpha$ is $(\log_{j+2}^+n)^\alpha$, so that iterates of $\Cf$ acting on $\Hw$ fit into the scope of this theorem. Second, it is natural to replace $\Da$ with $\Hw$ and  $w_n=d_{\alpha+1}(n)$ .  Here, when $\alpha$ is a positive integer, $d_\alpha(n)$ is the number of  representations of $n$ as a product of $\alpha$ integers, so $d_2(n)=d(n)$. For general $\alpha$, $d_\alpha(n)$ is the coefficient of the $n$th term of the Dirichlet series of $\zeta(s)^\alpha $, i.e.
\[ d_\alpha(n) =  \binom{ k_1+ \alpha -1 }{ k_1 } \cdots \binom{ k_r + \alpha -1 }{ k_r },\quad \text{ for $n=p_1^{k_1}\cdots p_r^{k_r}$ }.\] It can be checked that the proof of Theorem 1 of \cite{bb} carries through, so that $\Cf$ maps $\Hw$ with $w_n=d_{\alpha+1}(n)$ into $\mathscr H_\Omega$. Notice that $d_{\alpha+1}(n)$ has average order $(\log n)^\alpha$  \cite{anp} and $d(n)^\alpha$ has average order $(\log n)^{2^\alpha-1}$.\\

\section{Preliminaries}\label{pre}
In \cite{hls}, $\Ht$ was identified with a space of analytic functions on $\D^\infty\cap \ell^2(\N)$, where $\D^\infty$ is the infinite polydisk
\[\D^\infty := \left\{ z=(z_1,z_2,\dots),\quad |z_j|<1 \right\}.\]
This is obtained by using the Bohr lift of Dirichlet series that are analytic in $\C_{1/2}$, which is defined in the following way. Let
\begin{equation}\label{dseries} F(s)=\sum_{ n=1 }^\infty a_n n^{-s}, \quad\Re s>1/2.\end{equation}
We write $n$ as a product of its prime factors
\[ n =  p_1^{k_1} \cdots p_r^{k_r}, \]
where the $p_j$ are the primes in ascending order. We replace $p_j^{-s}$ with $z_j$, set $\kappa(n)=(k_1,\dots, k_r) $, and define the formal power series 
\[  \mathscr B F(z) = \sum_{n=1}^\infty a_nz^{ \kappa(n)  } \]
as the Bohr lift of $F$. 

For $\chi=( \chi_1,\chi_2, \dots  )\in\C^\infty$, we define a completely multiplicative function by requiring $\chi(n) = \chi^{\kappa(n)}$ when $n =   p_1^{k_1} \cdots p_r^{k_r}$ and $\kappa(n)=(k_1,\dots, k_r)$. For $\Phi$ of the form \eqref{phi} with $c_0=0$,
\[\Phi_{\chi}(s) = \phi_{\chi}(s) = \sum_{n=1}^\infty c_n \chi(n)n^{-s}.\]
\begin{lem}
Suppose that $\Phi\in\mathscr G$. Then $\Phi_\chi\in\mathscr G$ for any $\chi\in\overline{\D^\infty}$. 
\end{lem}
\begin{proof}This was proved in \cite[Lemma 8]{bb}. 
\end{proof}
\begin{lem}\label{fphichi}
Suppose that $\Phi\in\mathscr G$ of the form \eqref{phi} with $c_0=0$. For every Dirichlet polynomial $F$, every $\chi\in\overline{\D^\infty}$ and every $s\in\C_0$, we have
\begin{equation}\label{ffk}
(F\circ \Phi)_\chi (s)  = (F \circ \phi_\chi) (s).
\end{equation}
\end{lem}
\begin{proof}It was proved in \cite[Lemma 9]{bb} that \[(F\circ \Phi)_\chi (s) = (F_{\chi^{c_0}} \circ \Phi_\chi) (s) \] whenever $\Phi\in\mathscr G$. This is reduced to \eqref{ffk} when $c_0=0$.
\end{proof}
We shall now introduce a scale of Bergman spaces over $\D$, as well as the corresponding Bergman spaces over $\C_{1/2}$ which are induced by a certain conformal mapping $\tau:\C_{1/2} \to \D$. 

Let $e_j:=\underbrace{ \exp(\exp(\cdots\exp(e))) }_{ \text{ $j$ e's} }$ ($e_0=1$). For $\alpha>0$ and $j\geqslant 1$, we define
\begin{align*}\label{dvj}
dv_{\alpha,1}(z) := & d A_\alpha(z) = \alpha( 1- |z|^2 )^{ \alpha - 1 } dA(z);\\
\\
dv_{\alpha,j}(z) : = & \frac{ \alpha }{1-|z|^2} \left(\prod_{\ell=1}^{j-2} \log_\ell \left( \frac{e_\ell}{1-|z|^2} \right) \right)^{-1} \left( \log_{j-1} \left( \frac{e_{j-1}}{1-|z|^2} \right) \right)^{-(\alpha+1)}dA(z),\quad j>1.
\\                            
\end{align*}
Let $D_{\alpha,j}(\D)$ be the set of functions $f$ that satisfy 
\[ \| f \|_{D_{\alpha,j}(\D)}^2 := \int_{\D} | f(z) |^2 dv_{\alpha,j}(z) < \infty.  \]
For $f(z)=\sum_{n=0}^\infty c_n z^n$, we have \[ \| f \|_{D_{\alpha,j}(\D)}^2 \asymp  \sum_{n=0}^\infty \frac{ |c_n|^2 }{ \left( 1 + \log_{j-1}^+ n \right)^\alpha }. \]
Let \begin{equation}\label{tau}\tau (s) =  \frac{s-3/2}{s+1/2}\end{equation} 
\\
which maps $\C_{1/2}$ to $\D$.

The measure $\mu_j(s)$ on $\C_{1/2}$ induced by $\tau$ is
\\
\begin{align*}
d\mu_1(s) = & 4^\alpha \alpha \left( \sigma-1/2   \right)^{ \alpha-1 } \frac{ dm(s) }{ | s+1/2 |^{ 2\alpha+2 } };\\
\\
d\mu_j(s) = &  \frac{ \alpha }{ (\sigma-1/2) }\left(\prod_{\ell=1}^{j-2} \log_\ell^+  \frac{ e_\ell | s+1/2 |^2 }{ 2(2\sigma-1) }  \right)^{-1} \left( \log_{j-1}^+  \frac{ e_{j-1} | s+1/2 |^2 }{ 2(2\sigma-1) }  \right)^{-(\alpha+1)}\frac{ dm(s) }{  | s+1/2 |^2 } ,\quad j>1.
\end{align*}
Finally, let $D_{\alpha,j,i}(\C_{1/2})$ consist of functions $F$ that are analytic in $\C_{1/2}$ such that
\begin{equation}\label{dabeta} \| F \|_{ D_{\alpha,j,i}(\C_{1/2}) }^2 := \int_{\C_{1/2}} | F(s) |^2 d\mu_j(s) < \infty \end{equation}
and
\[ D_{\alpha,1,i}(\C_{1/2}) =: D_{\alpha,i}(\C_{1/2}). \]
Then
\[ \| F \|_{ D_{\alpha,j,i}(\C_{1/2}) }^2 =  \int_{\D} | F\circ \tau^{-1}(z) |^2 dv_{\alpha,j}(z) = \|F\circ \tau^{-1} \|_{D_{\alpha,j}(\D)}^2\]
and
\[ D_{\alpha,1}(\D) =:D_\alpha(\D). \]

The proof of the main theorem will be given in Section \ref{lsp}. We  verify that the average order of  $(\log_j^+\Omega(n))^\alpha$ is $(\log_{j+2}^+ n)^\alpha$ in Section \ref{averageorder}.

\section{Proof of Theorem \ref{main}}\label{lsp}
As in \cite[Subsection 3.1]{bb}, we inherit the proof of the arithmetical condition of $c_0$ from \cite[Theorem A]{gh}. For the mapping and convergence properties of $\phi$, we follow Subsection 3.2 in \cite{bb} as well.

\begin{lem}\label{ext}
Assume that $w_n\geqslant 1$. There exists a function $F\in \Hw$ such that
\begin{enumerate}
\item For almost all $\chi\in\T^\infty$, $F_\chi$ converges in $\C_0$ and cannot be analytically continued to any larger domain;
\item For at least one $\chi\in\T^\infty$, $F_\chi$ converges in $\C_{1/2}$ and cannot be analytically continued to any larger domain.
\end{enumerate}
\end{lem}
\begin{proof}It was shown in \cite{gh} that the function 
\[ F(s) =  \sum_{ p\ \text{prime} } \frac{1}{ \sqrt{ p } \log p } p^{-s}  \]
satisfies conditions $(1)$ and $(2)$. Clearly, $F$ is in $\Hw$ because $w_n\geqslant 1$.
\end{proof}
The rest of the proof consists of two steps. We shall first embed $\Hw$ into certain Bergman spaces, and then apply Littlewood's subordination principle to functions in these Bergman spaces.

\begin{lem}[Embedding of $\Hw$]\label{contiembed}
Let the weight $\{w_n\}$ of $\Hw$ have average order $( \log_j n )^\alpha$. Then $\mathscr{H}_{w}  $ is continuously embedded into $D_{\alpha,j,i}(\C_{1/2})$.
\end{lem}
For every $\tau\in\Z$, we define $Q_\tau=(1/2, 1]\times [\tau, \tau+1)$. It suffices to prove the following local embedding for $\Hw$,
\[\sup_{\tau\in\R} \int_{Q_\tau} |F(\sigma+it)|^2 dtd\mu_{j}^*(\sigma) \ll \|F\|_{\Hw}^2. \]
The case when $j=1$ was established in \cite{O}. We shall use the same method to establish the general case.

It will suffice to prove the inequality
\begin{equation}
\label{local}
\int_{1/2}^1 \int_0^1 \left| F(\sigma + it)\right|^2  dtd\mu_{j}^*(\sigma) \ll \|F\|_{\Hw}^2,
\end{equation}
where
\begin{align*}
d\mu_j^*(\sigma) : = &  \frac{ \alpha }{ (\sigma-1/2) }  \left( \prod_{\ell=1}^{j-2}\log_\ell^+  \frac{ e_\ell }{ (\sigma-1/2) }  \right)^{-1} \left( \log_{j-1}^+  \frac{ e_{j-1} }{ (\sigma-1/2) }  \right)^{-(\alpha+1)} d\sigma,\quad j> 1.\\
\end{align*}
We need the following lemma. 
\begin{lem}\label{weightom}
For $\alpha>0$ and $j\geqslant 2$, letting $n\to\infty$, we have
 \[ \int_0^1 n^{-t} \left( \prod_{l=1}^{j-2}\log_{l}^+ \frac{e_\ell}{t} \right)^{-1}  \left( \log_{ j-1 }^+ \frac{e_{j-1} }{t} \right)^{-(\alpha+1)} \frac{dt }{t} = \frac{1}{\alpha}  \left( \log_{j}^+ n  \right)^{-\alpha } + \mathscr{O}\left( \left( \log_{j}^+ n \right)^{-\alpha-1 } \right). \]
\end{lem}

\begin{proof}We first prove the case $j=2$ which is given by the integral
\begin{equation*}
I := \int_0^1 e^{-t\log n}\left( \log \frac{e}{t} \right)^{-(\alpha+1)} \frac{dt}{t}. \\
\end{equation*}
We split the integral at the point $t = 1/\log n$, which is dictated by the exponential decay of the integrand. This gives 
\[I = I_1 + I_2, \]
where 
\[I_1 = \int_0^{ 1/\log n } e^{-t\log n}\left( \log \frac{e}{t} \right)^{-(\alpha+1)} \frac{dt}{t}   \]
and
\[I_2 = \int_{ 1/\log n }^{ 1} e^{-t\log n}\left( \log \frac{e}{t} \right)^{-(\alpha+1)} \frac{dt}{t} .  \]
For $I_2$, we split it again at the point $t=\frac{1}{ \sqrt{ \log n } }$: 
\begin{align*}
 \int_{ 1/\log n }^{ 1/\sqrt{\log n}} e^{-t\log n}\left( \log \frac{e}{t} \right)^{-(\alpha+1)} \frac{dt}{t} \ll & \frac{1}{ \left( \log e\sqrt{\log n} \right)^{1+\alpha}  }\int_{ 1/\log n }^{ 1/\sqrt{\log n}} e^{-t\log n} \frac{dt}{t}\\
                                                                                                                                                    \ll & \frac{1}{ \left( \log_2 n \right)^{ \alpha +1 }  } \int_1^\infty e^{ -t } \frac{ dt }{ t } \\
                                                                                                                                                    \ll &  \frac{1}{ \left( \log_2 n \right)^{ \alpha +1 }  };
\end{align*}
\begin{align*}
 \int_{ 1/\sqrt{\log n}}^1 e^{-t\log n}\left( \log \frac{e}{t} \right)^{-(\alpha+1)} \frac{dt}{t} \ll & \sqrt{ \log n } \int_{ 1/\sqrt{\log n}}^1  e^{-t\log n} dt\\
                                                                                                                                      \\ \ll&\frac{e^{-\sqrt{\log n}  }}{ \sqrt{ \log n }} .
\end{align*}
For $I_1$, we write $e^{-t\log n}=1 + \mathscr{O}\left( t\log n  \right)$
\begin{align*} I_1 = & \int_0^{ 1/\log n } \left( 1  + \mathscr{O}\left( t\log n  \right) \right)   \left( \log \frac{e}{t} \right)^{-(\alpha+1)} \frac{dt}{t}   \\
                            = &  \int_{\log n}^\infty  \left( \log (et) \right)^{-(\alpha+1)} \frac{dt}{t}  + \mathscr{ O }\left( \log n \int_0^{ 1/\log n }  \left( \log \frac{e}{t} \right)^{-(\alpha+1)} dt \right) \\
                            = & \frac{1}{\alpha}  \left( \log_2 n  \right)^{-\alpha } + \mathscr{O}\left( \left( \log_2 n \right)^{-\alpha-1 } \right).
\end{align*}

For the general integral with $j> 2$ we can follow the same procedure. The main contribution comes from the term $I_1$, and $I_2$ gives a negligible contribution, that is
\begin{align*} 
I_1 = & \frac{1}{\alpha}  \left( \log_{j}^+ n  \right)^{-\alpha } + \mathscr{O}\left( \left( \log_{j}^+ n \right)^{-\alpha-1 } \right),\\
I_2\ll  & \left( \log_{j}^+ n \right)^{ -\alpha -1 }  + \frac{e^{-\sqrt{\log n}  }}{ \sqrt{ \log n }}.\\
\end{align*}

\end{proof}

\begin{proof}[Proof of Lemma \ref{contiembed}]Let $ F\in \Hw$. Using duality, we have

\begin{align*}
\int_0^1 \left| F(\sigma + it)\right|^2  dt = & \sup_{ g\in L^2(0,1) \atop \left\|g\right\|_2=1}\left| \int_0^1 \sum_{n\geqslant 1} a_n n^{-\sigma - it} g(t)dt\right| \\
                                             \leqslant &\underbrace{ \left( \sum_{n=1}^\infty \frac{ \left|  a_n\right|^2 }{ w_n } n^{-2\sigma +1} (\log_j^+ n)^\alpha \right) }_{ \text{(i)}} \underbrace{\left( \sup_{ g\in L^2(0,1) \atop \left\|g\right\|_2=1}\sum_{n=1}^\infty \frac{\left| \hat{g}(\log n)\right|^2 w_n }{n (\log_j^+ n)^\alpha }  \right)}_{ \text{(ii)} },\\                                             
\end{align*}
where $\hat{g}$ is the Fourier transform of $g$. By the smoothness of $\hat{g}$ and the assumption on $w_n$, the supremum on the right hand side is finite. Integrating both sides against $d\mu_{j}^*(\sigma)$ and applying Lemma \ref{weightom}, we get the inequality (\ref{local}).\\

\end{proof}

\begin{lem}\label{ subordv }
For $\omega: \D\to\D$ and $f\in D_{\alpha,j}(\D)$, there exists some constant $C_0$ depending on $\omega(0)$ such that
\[ \| f\circ \omega \|_{ D_{\alpha,j}(\D)  }^2 \leqslant C_0 \| f \|_{ D_{\alpha,j}(\D)  }^2 ,  \]
i.e.,
\begin{equation}\label{subord}
 \int_\D \left| f \left( \omega \left( z \right)   \right) \right|^2 dv_{\alpha,j}(z) \leqslant C_0\int_\D \left| f \left( z \right) \right|^2 dv_{\alpha,j}(z).
\end{equation}
\end{lem}
\begin{proof} 
Suppose $\omega(0) = a$, and let $\psi:\D\to\D$ be analytic such that $\psi(z)=\omega_a \circ \omega(z)$ with $\omega_a (z) = \frac{z-a}{1-\bar{a}z }$ . Then we have $\omega(z)=\omega_a \circ \psi(z)$. Starting with Littlewood's subordination principle,
\begin{align*}
\int_0^{2\pi} \left| f \left(  \omega \left( re^{i\theta}\right)   \right) \right|^2 \frac{d\theta}{2\pi} =  & \int_0^{2\pi} \left| f \left(  \omega_a \circ \psi \left( re^{i\theta} \right)   \right) \right|^2 \frac{d\theta}{2\pi} \\
                                                                                                       \leqslant & \int_0^{2\pi} \left| f \left(  \omega_a  \left( re^{i\theta} \right)   \right) \right|^2 \frac{d\theta}{2\pi}
\end{align*}
since $\psi(0)=0$.
Therefore,
\begin{align*}
& \int_\D \left| f \left( \omega ( z )   \right) \right|^2 dv_{\alpha,j}(z)  \\
&\quad\leqslant  \left( 1-|a|^2 \right) \int_\D \left| f \left( z \right) \right|^2  \left(\prod_{\ell=1}^{j-1} \log_\ell^+ \frac{e_\ell}{ 1 - | \omega_a(z) |^2 }  \right)^{-1} \left( \log_j^+ \frac{e_j  }{ 1 - | \omega_a(z) |^2 }  \right)^{-(\alpha+1)}  \frac{ dA(z) }{ \left(1-|z|^2\right)\left| 1-\bar{a}z \right|^2 }   \\
\\
&\quad\leqslant C_a \left( \frac{  1+|a| }{ 1-|a| } \right) \int_\D \left| f \left( z \right) \right|^2 dv_{\alpha,j}(z)                                                                                              
\end{align*}
\\
for some $C_a$ depending on $a$.\\
\end{proof}

\begin{proof}[Proof of Theorem \ref{main}]
When $c_0=0$, by Lemma \ref{fphichi}, $\left( F\circ\Phi \right)_\chi (s)  =  \left( F \circ\Phi_\chi \right) (s)$ for every Dirichlet polynomial $F$, $\chi\in\overline{\D^\infty}$ and $s\in\C_0$.  For fixed $s$, $\lambda\in \D$, $\chi\in\T^\infty $ and $\lambda\chi := ( \lambda\chi_1,  \lambda\chi_2,  \lambda\chi_3, \dots )$, we may view
\[ \Phi_{\lambda\chi}(s) = \sum_{n=1}^\infty c_n \lambda^{\Omega(n)} \chi(n)n^{-s} = \phi_{\lambda\chi}(s) :=\eta_s(\lambda) \]
as an analytic map $\eta_s: \D\to\C_{1/2}$ with $\eta_s(0)=c_1$. Putting $\omega = \tau\circ\eta_s = \tau\circ \Phi_{\lambda\chi}$ and applying $\tau$ to the inequality \eqref{subord} with $f=F\circ\tau^{-1}$ and being a Dirichlet polynomial, we have
\begin{equation}\label{fcomphi}
 \| F\circ\Phi_{\lambda\chi} \|_{D_{\alpha,j}(\D)}^2  \leqslant C_{1} \frac{ 1+\tau(c_1)  }{ 1-\tau(c_1) } \| F \|_{  D_{\alpha,j,i}(\C_{1/2}) }^2.  
\end{equation}
As in \cite{bb}, we assume $F\circ \Phi(s) = \sum_{n=2}^\infty b_nn^{-s}$. To avoid negative arguments in the $j$-fold logarithm, we shall equip $\chi\in\T^\infty$ with an indicator function with respect to the value of $\Omega(n)$ by defining
\[\chi_j(n) = \chi(n)\cdot \mathds{1}_{\big\{n:\ \Omega(n) > e_{j-3} \big\}} (n). \]
Then we put $F\circ\Phi_{\lambda\chi_j} = \sum_{n=2}^\infty b_n \lambda^{\Omega(n)} \chi_j(n)n^{-s}$ into \eqref{fcomphi} and integrate against the Haar measure $d\rho(\chi)$ over $\T^\infty$ to get
\begin{equation}
\label{cnull} 
\int_{\T^\infty}\int_\D \left| \sum_{n=1}^\infty b_n \lambda^{\Omega(n)} \chi_j(n)n^{-s}   \right|^2 dv_{\alpha,j}(\lambda)d\rho(\chi) \asymp  \sum_{\big\{n:\ \Omega(n) > e_{j-3} \big\}} \frac{ |b_n|^2 n^{-2\sigma} }{ \left( 1+\log_{j-1}\Omega(n) \right)^\alpha }.
\end{equation}
Upon letting $\sigma\to 0$ we have
\[  \sum_{\big\{n:\ \Omega(n) > e_{j-3} \big\}} \frac{ |b_n|^2 }{ \left( 1+\log_{j-1}\Omega(n) \right)^\alpha } \leqslant C_{c_1} \frac{ 1+\tau(c_1)  }{ 1-\tau(c_1) } \| F \|_{  D_{\alpha,j,i}(\C_{1/2}) }^2  \]
for some constant depending on $c_1$.
We get our conclusion by Lemma \ref{contiembed}.\\
\end{proof}

Even though we may get  $C_0=1$ in Lemma \ref{ subordv } by requiring $\omega(0) = 0 = \tau\circ\eta_s(0) = \tau \left(  c_1 \right) $, we cannot claim the contractivity due to the constant appearing in the embedding.

\section{The average order} \label{averageorder}
In this section, we will verify that the average order of  $(\log_j^+\Omega(n))^\alpha$ is $(\log_{j+2}^+ n)^\alpha$. It will suffice to give the details when $j=0$.
\begin{prop}\label{higherorder}
For  real $\alpha\geqslant1$, we have
\begin{equation}\label{oa}
\sum_{n\leqslant X} \Omega(n)^\alpha =  X \left( \log_2 X \right)^\alpha + \mathscr{O}\left[ X \left( \log_2 X  \right)^{\alpha-1}  \right]
\end{equation}
\end{prop}
This estimation is consistent with the case when $\alpha=1$ or $2$ which can be found in \cite{T}. 
Let \[ N_k(X) : = \# \big\{ n\leqslant X :  \Omega(n) = k  \big \} \] 
and 
\[ S_\Omega^\alpha(X) = \sum_{n\leqslant X} \Omega(n)^\alpha.   \]
We shall use some results of $N_k(X)$ to estimate $S_\Omega^\alpha( X )$, for which we need to rewrite $S_\Omega^\alpha(X)$ as
\[ S_\Omega^\alpha(X) =  \sum_{  1 \leqslant k\leqslant \frac{ \log X }{ \log 2 }  }k^\alpha N_k( X ).   \]
\begin{proof}[Proof of Proposition \ref{higherorder} ]
The quantity $N_k(X)$ has several changes in its behaviour as $k$ varies with $X$. These are described in \cite{T} (see Theorems 5 and 6 in Chapter II.6) and \cite{N}. Accordingly, we split the sum $S_\Omega^\alpha(X)$ into different parts with respect to $k$. These are given by:
\begin{enumerate}[label=(\roman*)]
\item $1 \leqslant k \leqslant (1-\epsilon)\log_2 X$;
\item $(1-\epsilon)\log_2 X < k < (1+\epsilon)\log_2 X$; 
\item $(1+\epsilon)\log_2 X \leqslant k < A\log_2 X$; 
\item $A\log_2 X \leqslant k \leqslant \frac{ \log X }{ \log 2 }$.
\end{enumerate}
Here  $A\geqslant 3$ and $\epsilon=\sqrt{2B\log_3 X/\log_2 X}$ with $B$ some large constant. With this choice of $\epsilon$ the sum over the range (ii) is centred about the mean of $\Omega(n)$ and hence should give the main contribution. We first concentrate on this range.   

Theorem 5 in Chapter II.6 of \cite{T} states that
\[ N_k(X) =  \frac{ X }{ \log X } \frac{ ( \log_2 X )^{k-1} }{ (k-1)! } \left\{ \nu \left( \frac{ k-1 }{ \log_2 X  } \right) + \mathscr{O} \left( \frac{ k }{ ( \log_2 X )^2 } \right)  \right\}, \]
where \[ \nu(z) := \frac{ 1 }{ \Gamma( z+1 ) } \prod_p \left( 1- \frac{ z }{ p }  \right)^{-1} \left( 1- \frac{1}{p} \right)^z, \quad |z| < 2.\]
Therefore,
\begin{align*}
 \sum_{  k = (1-\epsilon)\log_2 X}^{ (1+\epsilon)\log_2 X } k^\alpha N_k(X) 
 = &\frac{X}{\log X}  \sum_{  k = (1-\epsilon)\log_2 X}^{ (1+\epsilon)\log_2 X } k^\alpha  \frac{ (\log_2 X )^{k-1} }{ (k-1)! }\left( \nu\left( \frac{ k-1 }{ \log_2 X  } \right)  +\mathscr{O}\left( \frac{1}{\log_2 X } \right)   \right) \\
 = & \frac{X}{\log X\log_2 X} \sum_{  k = (1-\epsilon)\log_2 X}^{ (1+\epsilon)\log_2 X } k^{\alpha +1 } \frac{ \left(\log_2 X \right)^{k} }{ k! }\left( \nu\left( \frac{ k-1 }{ \log_2 X  } \right)  +\mathscr{O}\left( \frac{1}{\log_2 X } \right)   \right).
\end{align*}
Applying Stirling's formula gives the sum
\[
\frac{X}{\sqrt{2\pi}\log X\log_2 X}\left(1+\mathscr{O}\left( \frac{1}{\log_2 X } \right) \right) \sum_{  k = (1-\epsilon)\log_2 X}^{ (1+\epsilon)\log_2 X } k^{\alpha +1/2 } \frac{ \left(\log_2 X \right)^{k} e^k}{ k^k }\left( \nu\left( \frac{ k-1 }{ \log_2 X  } \right)  +\mathscr{O}\left( \frac{1}{\log_2 X } \right)   \right).
\]
\\

We now write
$k = \log_2X + {\ell}$ with
\[ \ell\in\left(-\epsilon{\log_2 X}, \epsilon{\log_2 X} \right)=\left(- \sqrt{2B\log_2 X\log_3 X}, \sqrt{2B\log_2 X\log_3 X} \right)\]
so that
\begin{align*}
\nu\left( \frac{ k-1 }{ \log_2 X  } \right)=&\nu(1)+\nu^\prime(1)\frac{\ell-1}{\log_2 X}+\frac{1}{2}\nu^{\prime\prime}(1)\left(\frac{\ell-1}{\log_2 X}\right)^2+O\left(\frac{1}{\log_2 X}\right)\\
=&1+\nu^\prime(1)\frac{\ell}{\log_2 X}+\frac{1}{2}\nu^{\prime\prime}(1)\left(\frac{\ell}{\log_2 X}\right)^2+O\left(\frac{1}{\log_2 X}\right)
\end{align*}
and
\begin{equation}\label{kalpha}
k^{\alpha+1/2}=(\log_2 X)^{\alpha+1/2}\left(1+\frac{\ell}{\log_2 X}\right)^{\alpha+1/2}=(\log_2 X)^{\alpha+1/2}\sum_{m=0}^\infty \binom{\alpha+1/2}{m}\left(\frac{\ell}{\log_2 X}\right)^m.
\end{equation}
Upon forming the product of these two series, our sum becomes
\begin{multline*}
\frac{X(\log_2 X)^{\alpha-1/2}}{\sqrt{2\pi}}\left(1+\mathscr{O}\left( \frac{1}{\log_2 X } \right) \right) \sum_{  \ell = -\epsilon\log_2 X}^{ \epsilon\log_2 X }  \frac{ \left(\log_2 X \right)^{\log_2 X+\ell} e^\ell}{ (\log_2 X+\ell)^{\log_2 X+\ell} }\\
\times\left(1+c_1\frac{\ell}{\log_2 X} +c_2\left(\frac{\ell}{\log_2 X}\right)^2+\mathscr{O}\left( \frac{1}{\log_2 X } \right)   \right).
\end{multline*}
for some coefficients $c_j$.

We  also expand the first factor of our sum as follows
\begin{align*}
\frac{ \left(\log_2 X \right)^{\log_2 X+\ell} e^\ell}{ (\log_2 X+\ell)^{\log_2 X+\ell} }=&e^{\ell-(\log_2 X+\ell)\log(1+\ell/\log_2 X)}\\
=&e^{-\frac{\ell^2}{ 2\log_2 X}} \left[ 1+ \sum_{m=1}^\infty \frac{1}{m!}\left( \sum_{n=3}^\infty \frac{\ell^n}{ n(n-1)\left(\log_2 X\right)^{n-1}}  \right)^m\right].
\end{align*}
Thus, our summand takes the form
\begin{multline}
e^{-\frac{\ell^2}{ 2\log_2 X}} \bigg[1+ c_1\frac{\ell}{\log_2 X} +c_2\left(\frac{\ell}{\log_2 X}\right)^2+\frac{\ell^3}{6(\log_2 X)^2} \\
 +\left( \frac{c_1}{6}\log_2 X + \frac{1}{12}  \right)\left( \frac{\ell}{\log_2 X} \right)^4  +\mathscr{O}\left(\frac{1}{  \log_2 X  }\right) \bigg].  
\end{multline}
Upon performing the sum we only retain the terms with an even power of $\ell$. Hence, we are led to compute  
\begin{multline*}
\frac{X(\log_2 X)^{\alpha-1/2}}{\sqrt{2\pi}}\left(1+\mathscr{O}\left( \frac{1}{\log_2 X } \right) \right) \sum_{  \ell = -\epsilon\log_2 X}^{ \epsilon\log_2 X } e^{-\frac{\ell^2}{ 2\log_2 X}} \bigg[1 +c_2\left(\frac{\ell}{\log_2 X}\right)^2+ \\
 +\left( \frac{c_1}{6}\log_2 X + \frac{1}{12}  \right)\left( \frac{\ell}{\log_2 X} \right)^4  +\mathscr{O}\left(\frac{1}{  \log_2 X  }\right) \bigg].   
\end{multline*}

On approximating the sum with an integral via the Euler--Maclaurin summation formula we gain an error term of order $\mathscr{O}\left(\frac{1}{\left(\log_2 X\right)^B}\right)$.
Then the leading term is given by
\begin{align*}
\int_{-\sqrt{2B\log_2 X\log_3 X}}^{ \sqrt{2B\log_2 X\log_3 X} } e^{ -\frac{t^2}{ 2\log_2 X}  }  dt = & \sqrt{2\pi\log_2 X} \left( 1 - \frac{2}{\sqrt{2\pi}} \int_{\sqrt{B\log_3 X}}^\infty e^{-u^2} du\right) \\
= & \sqrt{2\pi\log_2 X} \left( 1 + \mathscr O \left( \frac{1}{(\log_2 X)^B} \right)\right)
\end{align*}
since
\[ \int_a^\infty e^{-u^2} du \leqslant\frac{1}{a} \int_a^\infty ue^{-u^2} du \ll \frac{1}{a}e^{-a^2}. \]
For the higher powers of $\ell$ we use the formulae
\[ \int_{\infty}^\infty e^{-u^2}u^{2n}du =  \Gamma\left(n+\frac{1}{2}\right) 
\]
and
\[\int_a^\infty e^{-u^2} u^{2n}du \ll a^{2n-1}e^{-a^2} \]
which follow from
\[\int_a^\infty u^{2n}e^{-u^2} du \leqslant \frac{1}{a}\int_a^\infty u^{2n+1}e^{-u^2} du=\frac{1}{2}a^{2n-1}e^{-a^2} +\frac{n}{a}\int_a^\infty u^{2n-1}e^{-u^2}du.\]
These give
\begin{align*}
\int_{-\sqrt{2B\log_3X\log_2 X}}^{ \sqrt{2B\log_3X\log_2 X} } e^{ -\frac{t^2}{ 2\log_2 X}  } t^{2n}  dt = &\left( \sqrt{2\log_2 X}\right)^{2n+1} \int_{-\sqrt{B\log_3X}}^{\sqrt{B\log_3X}} e^{-u^2} u^{2n} du \\
=&\left( \sqrt{2\log_2 X}\right)^{2n+1}\Gamma(n+1/2)+O\left((\log_2 X)^{n+1/2-B}(\log_3 X)^{n-1/2}\right)
\end{align*}

Putting everything together gives
\[ \sum_{  k = (1-\epsilon)\log_2 X}^{ (1+\epsilon)\log_2 X } k^\alpha N_k(X) = X \left( \log_2 X \right)^\alpha +  \mathscr{O}\left( X \left( \log_2 X  \right)^{\alpha-1 } \right). \]
It should be clear from the above that with more precision one can obtain an asymptotic expansion to any required degree of accuracy.

For (i) and (iii), we shall use the Erd\H os--Kac theorem for $\Omega(n)$\cite{anp}, which states that
\begin{equation}\label{ek} 
\#\left\{ n\leqslant X : a \leqslant \frac{ \Omega(n)-\log_2 n }{ \sqrt{ \log_2 n } } \leqslant b \right\} \sim \frac{ X }{ \sqrt{ 2\pi } } \int_a^b e^{ -\frac{ t^2 }{ 2 } } dt.  
\end{equation}
This gives
\begin{align*}
\sum_{ k=1}^{(1-\epsilon) \log_2X  } k^\alpha N_k(X) \ll &  \left( \log_2 X \right)^\alpha\sum_{ k=1}^{(1-\epsilon) \log_2X  }  N_k(X) \\
                                                                                                          = & \left( \log_2 X \right)^\alpha \underbrace{\sum_{ \substack{ n\leqslant X \\  1 \leqslant \Omega(n) \leqslant (1-\epsilon)\log_2X   } } 1. }_{ =:S_{1,X} }\\ 
\end{align*}
Similarly,
\begin{align*}
\sum_{ k=(1+\epsilon)\log_2X }^{A\log_2X  } k^\alpha N_k(X) \ll & \left( \log_2 X \right)^\alpha \underbrace{\sum_{ \substack{ n\leqslant X \\  (1+\epsilon)\log_2X \leqslant \Omega(n) \leqslant A\log_2X   } } 1 .}_{ =:S_{3,X} }\\ 
\end{align*}
For $X$ large, there exits $n_{0,X}$ such that for $n\geqslant n_{0,X}$  
\[ -C_1 \sqrt{ \log_2 X } \leqslant \frac{ \Omega(n)-\log_2 n }{ \sqrt{ \log_2 n } } \leqslant C_2 \epsilon \sqrt{ \log_2 X }\]
and
\[ C_1' \epsilon \sqrt{ \log_2 X } \leqslant \frac{ \Omega(n)-\log_2 n }{ \sqrt{ \log_2 n } } \leqslant C_2' A \sqrt{ \log_2 X },\]
for some $C_1$, $C_2$, $C_1'$, $C_2'$. Therefore,

\[S_{1,X} \sim  \frac{ X }{ \sqrt{ 2\pi } } \int_{ C_1 \epsilon \sqrt{ \log_2 X } }^{ C_2 A \sqrt{ \log_2 X }} e^{ -\frac{ t^2 }{ 2 } } dt \ll X \sqrt{ \log_2X } e^{ -\frac{ t^2 }{ 2 }  } \big|_{t = C_1\epsilon \sqrt{ \log_2 X }} \ll \frac{ X }{ \left(  \log X \right)^{\epsilon_0} }\]
and
\[S_{3,X}  \ll \frac{ X }{ \left(  \log X \right)^{\epsilon_0'} }\]
for some $\epsilon_0$, $ \epsilon_0'>0 $, respectively.\\

For (iv), by Nicolas's result in \cite{N}, there exists the same constant $C$, such that
 \begin{align*}
\sum_{ k=A\log_2X }^{ \log X/\log 2  } k^\alpha N_k(X) \ll &   \sum_{ k=A\log_2X }^{ \log X/\log 2  } k^\alpha \left\{ \left( \frac{C X }{2^k} \right) \log \left( \frac{X}{ 2^k } \right)  \right\} \\
                                                                                      = & CX \sum_{ k=A\log_2X }^{ \log X/\log 2  } k^\alpha2^{-k} \log \left( \frac{X}{ 2^k } \right).
\end{align*}
We may bound this last sum from above by the integral
\begin{align*}
\int_{ A\log_2X }^{ \log X/\log 2 } t^\alpha 2^{-t} \log \left( \frac{ X }{ 2^t } \right) dt & \leqslant \log X \int_{ A\log_2X }^\infty t^\alpha 2^{-t} dt \\
                                                                                                                                & \leqslant \log X \left( A\log_2 X \right)^\alpha 2^{- A\log_2X } \\
                                                                                                                                & \ll \log X \left( \log_2 X \right)^\alpha \left( \log X  \right)^{ -A\log 2 } \\
                                                                                                                                & \leqslant \left( \log X \right)^{ 1-3\log 2 } \left( \log_2 X \right)^\alpha.
\end{align*}

\end{proof}
For the average order of $\big( \log_j^+\Omega(n) \big)^\alpha$, when carrying through the proof of Proposition \ref{higherorder} for the range (ii), it can be seen from \eqref{kalpha} that the main contribution $\left(\log_{j+2}X\right)^\alpha$ gets more and more centralised when $j$ becomes bigger.
Therefore, we eventually get
\[ \sum_{n\leqslant X }\big( \log_j^+\Omega(n) \big)^\alpha =  X \left( \log_{j+2} X \right)^\alpha + \mathscr{O}\left( X \frac{ \left( \log_{j+2}X\right)^\alpha }{ \log_2 X } \right).\]
      
\section*{acknowledgements}
I thank Ole Fredrik Brevig for introducing me to this topic and giving valuable information during my work. I also thank Winston Heap and Kristian Seip for crucial advice on the proof of the asymptotic approximation. I thank all of them for their kind suggestions on early drafts of this paper.


\begin{thebibliography}{BRSHZE}
\addcontentsline{toc}{chapter}{Bibliography}
\bibitem{MB} M. Bailleul, \emph{ Composition operators on weighted Bergman spaces of Dirichlet series } J. Math. Anal. Appl. \textbf{426} (2015), no. 1, 340-363.

\bibitem{bb} M. Bailleul, O. F. Brevig, \emph{The composition operators on Bohr-Bergman spaces of Dirichlet series}, Ann. Acad. Sci. Fenn. Math. \textbf{41} (2016), no. 1, 129-142. 

\bibitem{Bay02} F. Bayart, \emph{Hardy spaces of Dirichlet Series and their composition operators,} Monatsh. Math. \textbf{136} (2002), no. 3, 203-236.

\bibitem{bayb}F. Bayart, O. F. Brevig, \emph{Composition operators and embedding theorems for some function spaces of Dirichlet series}, arXiv:1602.03446.

\bibitem{gh} J. Gordon, H. Hedenmalm, \emph{The composition operators on the space of Dirichlet series with square summable coefficients}, Michigan Math. J. \textbf{46} (1999), no. 2, 313-329.

\bibitem{HR} G. H. Hardy, S. Ramanujan, \emph{The normal number of prime factors of a number $n$,} Quart. J. Math. \textbf{48} (1917), 76-92.


\bibitem{hls} H. Hedenmalm, P. Lindqvist and K. Seip, \emph{A Hilbert space of Dirichlet series and systems of dilated functions in $ L^ 2 (0, 1) $}, Duke Math. J. \textbf{86} (1997), no. 1, 1-37.

\bibitem{N} J.--L. Nicolas, \emph{Sur la distribution des nombres entiers ayant une quantit\'e fix\'ee de facteurs premiers,} Acta Arith. \textbf{44} (1984), no. 3, 191-200. 

\bibitem{O} J.--F. Olsen, \emph{Local properties of Hilbert spaces of Dirichlet series}, J. Funct. Anal. \textbf{261} (2011), no. 9, 2669-2696.


\bibitem{strco} H. Queff\'elec, K. Seip, \emph{Approximation numbers of composition operators on the $\Ht$ space of Dirichlet series}, J. Funct. Anal. \textbf{268} (2015), no. 6, 1612-1648. 



\bibitem{anp} G. Tenenbaum, \emph{Introduction to Analytic and Probabilistic Number Theory,} Cambridge University Press.  \textbf{46} (1995).

\bibitem{T} P. Tur\'an, \emph{On a theorem of Hardy and Ramanujan,} J. London Math. Soc. S1-9 (1934), no. 4, 274.

\bibitem{Z} A. Zygmund,  \emph{Trigonometric series,} Cambridge University Press. (2002).

























\end{thebibliography}
\end{document}